\documentclass{amsart}
\usepackage{amsmath,amsthm,amssymb,graphicx}

\newtheorem{theorem}{Theorem}[section]
\newtheorem{lemma}[theorem]{Lemma}

\newtheorem{proposition}[theorem]{Proposition}
\newtheorem{definition}[theorem]{Definition}

\theoremstyle{remark}

\newtheorem{example}[theorem]{Example}

\begin{document}

\title{Algebraic discrete Morse theory for the hull resolution}
\author{Patrik Nor\'en}
\address{Department of Mathematics; Box 8205, NC State University; Raleigh, NC 27695-8205, U.S.A.}
\email{pgnoren2@ncsu.edu}

\begin{abstract}
We study how powerful algebraic discrete Morse theory is when applied to hull resolutions. The main result describes all cases when the hull resolution of the edge ideal of the complement of a triangle-free graph can be made minimal using algebraic discrete Morse theory.
\end{abstract}
\maketitle

\section{Introduction}

Finding minimal free resolutions of ideals is an important problem in commutative algebra. Cellular resolutions provide one of the main techniques for obtaining free resolutions of monomial ideals. A particularly nice type of cellular resolution is the hull resolution. The hull resolution preserves all the symmetry of the ideal itself but it is not necessarily minimal. Algebraic discrete Morse theory is a general method for making free resolutions smaller, but often it is not clear how powerful this method is. This paper studies the case when the monomial ideal is the edge ideal of a graph and algebraic discrete Morse theory is applied to the hull resolution.

The main result is the following.

\begin{theorem}\label{thm:main}
Let $\overline{G}$ be a triangle-free graph and let $I_G$ be the edge ideal of $G$. The Hull resolution of $I_G$ can be made minimal using algebraic discrete Morse theory if and only if $\overline{G}$ has no induced subgraph isomorphic to two disjoint cycles.
\end{theorem}

An important example is the complement of a cycle. There are a few cellular resolutions of the edge ideal of the complement of a cycle already in the literature. Bierman \cite{B}, Dochterman \cite{D} and Sturgeon \cite{S} give different constructions. As there is some choice in how to apply discrete Morse theory, the exact description of the resolution from the theory in this paper varies. However it is immediate that the cells correspond to components of induced subgraphs of the cycle and each induced proper subgraph gives one cell less than the number of components.

The minimal example where the algebraic discrete Morse theory can not be used to make the hull resolution minimal for complements of triangle-free graphs is the complement of two disjoint cycles. In these cases it is possible to get a cellular resolution with one single cell too many to be minimal.

The outline of the paper is as follows. Section \ref{Sec:Graph} introduces some basic graph theory concepts and notation. In Section \ref{Sec:polytope} edge ideals and the associated edge polytopes that support their hull resolutions are described. In Section \ref{Sec:Cells} the machinery of algebraic discrete Morse theory and cellular resolutions is briefly explained. In Section \ref{Sec:Hull} the hull resolutions for general edge ideals are explored. Section \ref{Sec:Comp} and Section \ref{Sec:All} are devoted to the combinatorics of the hull resolutions of the edge ideals of triangle-free graphs. In Section \ref{Sec:Main} Theorem \ref{thm:main} is proved.

\section{Graph theory}\label{Sec:Graph}

Graphs are finite and have no loops or multiple edges unless explicitly said otherwise.

Some basic notions from graph theory are needed. Let $G$ be a graph, the vertex set of $G$ is denoted $V(G)$ and the edge set of $G$ is denoted $E(G)$. The \emph{degree} of a vertex $v\in V(G)$ is the number of edges containing $v$ in $E(G)$. Let $V_i(G)$ be the set of vertices in $V(G)$ with degree $i$. A \emph{subgraph} of $G$ is a graph $G'$ with $V(G')\subseteq V(G)$ and $E(G')\subseteq E(G)$. A \emph{component} of $G$ is a maximal connected subgraph with nonempty vertex set. A set $U\subseteq V(G)$ is \emph{independent} if there is no edge between any pair of vertices in $U$. A graph is \emph{triangle-free} if it has no subgraph isomorphic to a triangle.

There are a few standard ways to construct new graphs that will be important. The \emph{complement} of a graph $G$ is denoted $\overline{G}$ and it is the graph  with $V(\overline{G})=V(G)$ so that two vertices are adjacent in $\overline{G}$ if and only if they are not adjacent in $G$.

If $U\subseteq V(G)$ then the subgraph of $G$ \emph{induced} by $U$ is denoted $G[U]$, it is the graph with $V(G[U])=U$ and vertices are adjacent in $G[U]$ if and only if they are adjacent in $G$. If $v\in V(G)$ then the graph obtained from $G$ by removing $v$ is $G\setminus v=G[V(G)\setminus\{v\}]$. 

Given two graphs $G_1$ and $G_2$ the graph $G_1\cap G_2$ is defined as the graph with $V(G_1\cap G_2)=V(G_1)\cap V(G_2)$ and $E(G_1\cap G_2)=E(G_1)\cap E(G_2)$, similarly for unions $V(G_1\cup G_2)=V(G_1)\cup V(G_2)$ and $E(G_1\cup G_2)=E(G_1)\cup E(G_2)$.

Stars are important special kinds of graphs.

\begin{definition}
A graph $G$ is a \emph{star with central vertex $v$} if $v\in V(G)$ and $E(G)=\{uv\mid u\in V(G)\setminus\{v\}\}$.
\end{definition}

Given a graph $G$ there are some special important subgraphs that will occur.

\begin{definition}
For every vertex $v\in V(G)$ define graphs $N_G(v)$ by  $V(N(v))=\{v\}\cup\{u\in V(G)\mid uv\in E(G)\}$ and $E(N_G(v))=\{uv\in E(G)\}$. For nonempty independent subsets $U$ of $V(G)$ define $N_G(U)=\cup_{v\in U}N_G(U)$. For edges $uv\in E(\overline{G})$ define $N_G(uv)=N_G(u)\cup N_G(v)$. Also define graphs $C_G(x)=G[V(G)\setminus V(N_G(x))]$ and $NC_G(x)=N_G(x)\cup C_G(x)$.
\end{definition}

Two important properties is that $N_G(v)$ is a star for vertices $v$ and the neighborhood of $v$ is the same in $G$ and $N_G(v)$. In fact this is a simple reformulation of the definition.

Given two different nonempty independent sets $U$ and $U'$  it will be important to know if $NC_G(U)\cap NC_G(U')$ has isolated vertices. In general $NC_G(U)\cap NC_G(U')$ is
\[(N_G(U)\cap N_G(U'))\cup(N_G(U)\cap C_G(U'))\cup(C_G(U)\cap N_G(U'))\cup(C_G(U)\cap C_G(U'))\]
were the unions are disjoint. In Section \ref{Sec:All} there are many propositions where the proofs depend on determining if the graph $NC_G(U)\cap NC_G(U')$ has isolated vertices or not, and the usual argument is to handle each part $(N_G(U)\cap N_G(U'))$, $(N_G(U)\cap C_G(U'))$, $(C_G(U)\cap N_G(U'))$ and $(C_G(U)\cap C_G(U'))$ separately.

Let $G$ be a graph and let $S\subseteq E(G)$ be a subset so that every vertex in $V_1(G)\cup V_2(G)$ is in at least one edge in $S$. There is a graph $F(G,S)$ that will be important in understanding the combinatorics of the hull resolutions of edge ideals.

The construction of $F(G,S)$ is in two steps. Construct the graph $F'(G,S)$ from $G$ by subdividing every edge in $S$. The vertex set of $F'(G,S)$ can be thought of as the union $V(F'(G))=V(G)\cup S$. In this way two elements in $S$ are adjacent if they have a vertex in common. Elements in $V(G)$ are adjacent if they were adjacent in $G$ but the edge between them is not in $S$. Finally a vertex $v$ in $V(G)$ is adjacent to an element $e$ in  $S$ if $v$ is one of the endpoints of $e$.

Construct $F(G,S)$ from $F'(G,S)$ by contracting one edge $ve$ for every $v\in V_1(G)\cup V_2(G)$ where $e\in S$. As the degree of $v$ does not change by subdividing edges the degree in $F'(G,S)$ is still one or two, in particular the combinatorics of the graph $F(G,S)$ do not depend on the choice of edge. An explicit description of the vertex set and adjacency in $F(G,S)$ will be useful. 

The vertex set of $F(G,S)$ is $S\cup V(G)\setminus (V_1(G)\cup V_2(G))$. Two edges in $S$ are adjacent if they have a common endpoint in $V_1(G)\cup V_2(G)$. Two vertices in $V(G)\setminus (V_1(G)\cup V_2(G))$ are adjacent if they are adjacent in $G$ but the edge between them is not in $S$. A vertex $v$ in $V(G)\setminus (V_1(G)\cup V_2(G))$ is adjacent to any edge in $S$ containing $v$. Finally a vertex $u$ in $V(G)\setminus (V_1(G)\cup V_2(G))$  is adjacent to the edge $vw$ in $S$ not containing $u$ if $v\in V_1(G)\cup V_2(G)$ and $uv\in E(G)\setminus S$.

It is immediate from the construction that $G$ and $F(G,S)$ are homotopy equivalent as topological spaces and this property will be important. Later there will be some operations on the topological spaces realized by graphs, for example contracting edges. This is the only situation where loops and multiple edges might occur.

Most graphs considered are undirected. In fact the only directed graphs occuring will be Hasse diagrams of posets, and graphs obtained from Hasse diagrams by reversing some edges. Recall that the vertices of the Hasse diagram of a poset is the elements of the poset and there is an edge from $u$ to $v$ if and only if $u>v$ and there is no element $w$ so that $u>w>v$. The \emph{dualization} of a poset $P$ is the poset whose Hasse diagram is obtained from the Hasse diagram of $P$ by reversing all edges.

\section{Edge ideals and edge polytopes}\label{Sec:polytope}

Let $G$ be a graph. Let $\{\mathbf{e}_v\mid v\in V(G)\}$ be the standard basis of $\mathbb{R}^{V(G)}$. The polytope $P_G$ obtained as the convex hull of $\{\mathbf{e}_i+\mathbf{e}_j\mid ij\in E(G)\}$ is \emph{the edge polytope of $G$.} It is immediate from the definition that the polytopes $P_G$ behave well with respect to the intersection operation on graphs $P_{G_1\cap G_2}=P_{G_1}\cap P_{G_2}$.

In order to give the facet description of $P_G$ some more notation is needed.
\begin{definition}
A vertex $v\in V(G)$ is \emph{ordinary} if the graph $G\setminus v$ is connected.
\end{definition}

\begin{definition}
A vertex $v\in V(G)$ is \emph{regular} if no component of $G\setminus v$ is bipartite.
\end{definition}

\begin{definition}
A nonempty independent set $U$ in a graph $G$ is an \emph{acceptable set in $G$} if both $N_G(U)$ and $C_G(U)$ are connected and $E(C_G(U))\neq \emptyset$.
\end{definition}

\begin{definition}
A nonempty independent set $U$ in a graph $G$ is a \emph{fundamental set in $G$} if $N_G(U)$ is connected and no component of  $C_G(U)$ is bipartite.
\end{definition}

The following two propositions are special cases of Theorem 1.7 in \cite{OH} by Ohsugi and Hibi.

\begin{proposition}\label{Prop:Bipartite}
Let $G$ be a connected bipartite graph with $E(G)\neq \emptyset$. The dimension of $P_G$ is $|V(G)|-2$ and the set of facets of $P_G$ is $\{P_{G\setminus v}\mid v$ is an ordinary vertex in $G\}\cup\{P_{NC_G(U)}\mid U$ is an acceptable set in $G\}$.

If $U\neq U'$ are two acceptable sets in $G$ then $P_{NC_G(U)}\neq P_{NC_G(U')}$.
\end{proposition}

\begin{proposition}\label{Prop:NotBipartite}
Let $G$ be a connected not bipartite graph. The dimension of $P_G$ is $|V(G)|-1$ and the set of facets of $P_G$ is $\{P_{G\setminus v}\mid v$ is a regular vertex in $G\}\cup\{P_{NC_G(U)}\mid U$ is a fundamental set in $G\}$.

If $U\neq U'$ are two fundamental sets in $G$ then $P_{NC_G(U)}\neq P_{NC_G(U')}$.
\end{proposition}

Induced subgraphs give faces in $P_G$.

\begin{proposition}\label{Prop:Face}
Let $G$ be a graph and let $U$ be a nonempty subset of $V(G)$. Let $H$ be the linear subspace of $\mathbb{R}^{V(G)}$ spanned by $\{\mathbf{e}_v\mid v\in U\}$. The polytope $P_{G[U]}$ is a face of $P_G$ and $P_{G[U]}=H\cap P_G$.
\end{proposition}

\begin{proof}
The vertices of $P_G$ in $H$ are exactly the points $\mathbf{e}_i+\mathbf{e}_j$ so that $ij\in E(G)$ and $\{i,j\}\subseteq U$. The intersection is a face as $P_G$ is contained in $[0,1]^{V(G)}$.
\end{proof}

Let $\mathbb{K}$ be a field. The ideal $I_G=\langle x_ix_j\mid ij\in E(G)\rangle\subseteq\mathbb{K}[x_v\mid v\in V(G)]$ is \emph{the edge ideal of $G$.} Define a map $\ell_G$ from the set of nonempty faces of $P_G$ to the monic monomials in $I_G$ by $\ell_G(\sigma)=$lcm$(x_ix_j\mid \mathbf{e}_i+\mathbf{e}_j\in \sigma)$. Sometimes it is useful to extend the domain of $\ell_G$ to include the empty set and then $\ell_G(\emptyset)=1$, in this case the range of $\ell_G$ is also extended. Let $M_G$ be the subposet of the face poset of $P_G$ consisting of all faces $\sigma$ with $\ell_G(\sigma)=\prod_{v\in V(G)}x_v$.

If $G$ has no edges the definitions are a bit degenerate, the conventions $P_G=\emptyset$, $I_G=\langle 1\rangle$, $\ell_G(\emptyset)=1$ and $M_G=\{\emptyset\}$ are sometimes used if $E(G)=\emptyset$.

\section{Discrete Morse theory and cellular resolutions}\label{Sec:Cells}

The machinery of cellular resolutions is a powerful tool used to construct free resolutions of monomial ideals. A cellular resolution of a monomial ideal $I$ is encoded by a cell complex $X$ and a labeling map $\ell$ from the set of cells of $X$ to $I$, the map $\ell$ have to satisfy $\ell(\sigma)=$lcm$(\ell(v)\mid v$ is a vertex of $\sigma)$. There is an easy condition for when a pair $X$ and $\ell$ gives a cellular resolution of $I$, the condition is that the image of $\ell$ generates $I$ and the subcomplex consisting of all cells with labels dividing a monomial $m$ is acyclic for all $m$. The condition for when cellular resolutions is minimal is that the resolution is minimal if and only if no cell is on the boundary of a cell with the same label.

It is always possible to construct a cellular resolution for a given monomial ideal, one construction is the hull resolution by Bayer and Sturmfels \cite{BS}. In the special case when $I$ is the edge ideal of a graph $G$ then the cell complex in the hull resolution is $P_G$ and the labeling map is $\ell_G$. In general it is not possible to give a minimal cellular resolution but algebraic discrete Morse theory can be used to make many cellular resolutions smaller.

The discrete Morse theory developed by Forman \cite{F} provides a way to reduce the number of cells in a CW-complex without changing the homotopy type.

There are  a few different ways to express discrete Morse theory, the way that works best for the algebraic setting is in terms of acyclic matchings in the Hasse diagram of the face poset of  the complex. Let $D$ be a directed graph. A subset $M\subseteq E(D)$ is a \emph{matching} if every vertex is in at most one of the edges in $M$. A matching is \emph{acyclic} if the graph obtained by reversing the edges in the matching contain no directed cycles. An important property of Hasse diagrams of a posets is that they contain no directed cycles. Given an acyclic matching $M$ of $D$ the elements of $V(D)$ that are not matched are \emph{critical}.

The main theorem of discrete Morse theory \cite{F} can be stated as follows.

\begin{theorem}
Let $X$ be a regular CW-complex with face poset $P$. If $M$ is an acyclic matching of $P$ where the empty face is critical, then there is a CW-complex $\tilde{X}$ homotopy equivalent to $X$. The critical cells are in bijection with the cells of $\tilde{X}$, this bijection preserve dimension.
\end{theorem}

For one-dimensional complexes the theory is greatly simplified and it is always possible to find optimal matchings in the sense that the resulting complex have the minimal number of cells of any complex homotopy equivalent to the original complex. One-dimensional complexes are essentially graphs where loops and multiple edges are allowed, the complexes obtained from discrete Morse theory are the complexes obtained by contracting non-loop edges. The matchings are pairings of a vertex with an edge containing the vertex, and the matched edge is then contracted and the new vertex is identified with the endpoint of the contracted edge not paired to the contracted edge. In particular it is possible to contract edges in a graph until there is only a single vertex in each component and there is a matching realizing this. The space of acyclic matchings for the Hasse diagram of posets of one-dimensional complexes has interesting structure and was further studied by Chari and Joswig \cite{CJ}.

Batzies and Welker \cite{BW} extended discrete Morse theory to work well with cellular resolutions.

Let $X$ be a CW-complex with labeling map $\ell$ and face poset $P$. An acyclic matching $M$ of the Hasse diagram of $P$ satisfying $\sigma\tau\in M\Rightarrow \ell(\sigma)=\ell(\tau)$ is \emph{homogenous}, that is the matching is homogenous if cells are only matched to cells with the same label.

The main theorem of algebraic discrete Morse theory for cellular resolutions \cite{BW} can be stated as follows.

\begin{theorem}\label{thm:adm}
Let $X$ be a regular CW-complex with face poset $P$. Let $\ell$ be a labeling of $X$ giving a cellular resolution of the ideal $I$. If $M$ is a homogenous acyclic matching of $P$ then $\tilde{X}$ also supports a cellular resolution of $I$. The cell corresponding to the critical cell $\sigma$ has label $\ell(\sigma)$.
\end{theorem}

\section{Hull resolutions of edge ideals}\label{Sec:Hull}

A first step to understand the hull resolution of $I_G$ is to understand the set of cells with a given label.

\begin{proposition}\label{Prop:Isolated}
If $G$ is a graph with $E(G)\neq\emptyset$ then $\ell_G(P_G)=\prod_{v\in V(G)}x_v$ if and only if $G$ has no isolated vertex. 
\end{proposition}

\begin{proof}
If $G$ has an isolated vertex $v$ then $x_v$ does not divide any of the generators of $I_G$ and then $x_v$ does not divide any monomial in the image of $\ell_G$. If $G$ has no isolated vertex then for every vertex $v\in V(G)$ there is some edge $uv\in E(G)$, in particular $x_vx_u$ divides $\ell_G(P_G)$.
\end{proof}

It is possible to describe the image of $\ell_G$.

\begin{proposition}\label{Prop:Image}
Let $G$ be a graph and let $U$ be a nonempty subset of $V(G)$. The monomial $\prod_{v\in U}x_v$ is in the image of $\ell_G$ if and only if $G[U]$ has no isolated vertex. Furthermore if $G[U]$ has no isolated vertex then $P_{G[U]}$ has label $\prod_{v\in U}x_v$ and all other faces of $P_G$ with this label are contained in $P_{G[U]}$.
\end{proposition}

\begin{proof}
Any face with label $\prod_{v\in U}x_v$ has to be contained in the subspace $H$ in Proposition \ref{Prop:Face}, this proves that any face with the desired label has to be contained in $P_{G[U]}=H\cap P_G$. Now $\ell_G(P_{G[U]})=\ell_{G[U]}(P_{G[U]})=\prod_{v\in U}x_v$ if and only if $G[U]$ has no isolated vertex by Proposition \ref{Prop:Isolated}.
\end{proof}

One useful aspect of Proposition \ref{Prop:Image} is that it makes it possible to think of the set of faces with label $\prod_{v\in U}x_v$ as the set of faces with the maximal label for some hopefully smaller graph.

When $G$ is disconnected then the behavior of the label can be understood in terms of the components.

\begin{proposition}\label{Prop:Disconnect}
Let $G$ be the disjoint union of the connected graphs $G_1,\ldots,G_n$ and let each $G_i$ have at least one edge. The polytope $P_G$ is a realisation of the join $*_{i\in [n]} P_{G_i}$ where the label satisfies $\ell_G(*_{i\in [n]} \sigma_i)=\prod_{i\in[n]}\ell_{G_i}(\sigma_i)$ with $\ell_{G_i}(\emptyset)=1$.
\end{proposition}

\begin{proof}
The polytopes $P_{G_i}$ are contained in mutually orthogonal and nonintersecting affine subspaces of $\mathbb{R}^{V(G)}$, furthermore $P_G$ is the convex hull of the union of faces $\cup_{i\in [n]}P_{G_i}$ and then $P_G$ is the indicated join.

The formula for the label is true by definition for the vertices of $P_G$ and the general case follows as no variable that divides $\ell_{G_i}(\sigma_i)$ can divide $\ell_{G_j}(\sigma_j)$ for $i\neq j$.
\end{proof}

Now the posets of faces with a given label can be described.

\begin{theorem}\label{Prop:Describe}
Let $G$ be a graph and let $U$ be a nonempty subset of $V(G)$ so that $G[U]$ has no isolated vertex. If $G[U]$ is the disjoint union of the nonempty connected graphs $G_1,\ldots,G_n$ then each $G_i$ contains an edge. The subposet of the face poset of $P_G$ consisting of all cells with label $\prod_{v\in U}x_v$ is isomorphic to $\prod_{i\in [n]}M_{G_i}$.
\end{theorem}

\begin{proof}
Proposition \ref{Prop:Image} shows that it is enough to consider the case $U=V(G)$. The face poset of a join of polytopes is the product of the face posets and then the result follows from Proposition \ref{Prop:Disconnect}. 
\end{proof}

Using the facet descriptions of $P_G$ it is possible to understand the set of facets in $M_G$ for connected $G$.

\begin{proposition}
Let $G$ be a connected graph with $E(G)\neq\emptyset$. Let $v\in V(G)$ be an ordinary vertex if $G$ is bipartite and let $v\in V(G)$ be a regular vertex if $G$ is not bipartite. The facet $P_{G\setminus v}$ of $P_G$ is not in $M_G$.
\end{proposition}

\begin{proof}
The label of $P_{G\setminus v}$ is not divisible by $x_v$ as $v$ is not a vertex of $G\setminus v$.
\end{proof}

\begin{proposition}\label{Prop:Facet}
Let $G$ be a connected graph with $E(G)\neq\emptyset$. If $G$ is not bipartite and $U$ is a fundamental set in $G$ then $P_{NC_G(U)}$ is in $M_G$. If $G$ is bipartite and $U$ is an acceptable set in $G$ then $P_{NC_G(U)}$ is in $M_G$.
\end{proposition}

\begin{proof}
Proposition \ref{Prop:Isolated} says that the face $P_{NC_G(U)}$ is in $M_G$ if and only if $NC_G(U)$ has no isolated vertex. No component of $C_G(U)$ is an isolated vertex by definition of acceptable and fundamental. The fact that no vertex is isolated in $G$ proves that no vertex is isolated in $N_G(U)$.
\end{proof}

It is also possible to understand faces with lower dimension in $M_G$.

\begin{proposition}\label{Prop:Edge}
Let $G$ be a connected graph with $E(G)\neq\emptyset$. Let $U_1,\ldots, U_n$ be fundamental sets in $G$ if $G$ is not bipartite and let $U_1,\ldots, U_n$ be acceptable sets in $G$ if $G$ is bipartite. The face $\cap_{n\in[n]}P_{NC_G(U_i)}$ is in $M_G$ if and only if $\cap_{n\in[n]}NC_G(U_i)$ has no isolated vertices.
\end{proposition}

\begin{proof}
The face $\cap_{n\in[n]}P_{NC_G(U_i)}$ is $P_{\cap_{i\in[n]}NC_G(U_i)}$ and the vertex set of the graph $\cap_{i\in[n]}NC_G(U_i)$ is $V(G)$. Now the statement follow from Proposition \ref{Prop:Isolated}.
\end{proof}

\section{Complements of triangle-free graphs}\label{Sec:Comp}

In general finding the set of independent sets of a graph is itself a challenging problem. We restrict our attention to the easier case when $G$ is the complement of a triangle-free graph $\overline{G}$. As $\overline{G}$ is triangle-free the set of independent sets of $G$ is  $\{\emptyset\}\cup\{\{v\}\mid v\in V(G)\}\cup\{\{u,v\}\mid uv\in E(\overline{G})\}$.

An edge $uv$ in $E(\overline{G})$ is \emph{fundamental} if $\{u,v\}$ is fundamental in $G$, denote the set of fundamental edges $S_G$. A vertex $v\in V(G)$ is \emph{fundamental} if $\{v\}$ is fundamental in $G$.

The acceptability concept can also be extended, but acceptability is only relevant for bipartite graphs and there is only a handful of bipartite graphs with triangle-free complements.

\begin{proposition}
If $G$ is a connected bipartite graph so that $\overline{G}$ is triangle-free then $G$ is a subgraph of the cycle with four vertices.
\end{proposition}

\begin{proof}
Neither part in the bipartition can have more than two vertices.
\end{proof}

The following list of examples explains $M_G$ for all connected subgraphs of the cycle with four vertices.

\begin{example}
If $G$ has no edges then $M_G$ is $\{\emptyset\}$ by definition.
\end{example}

\begin{example}
If $G$ is a path with one edge then $P_G$ is a point and $M_G$ only contains $P_G$.
\end{example}

\begin{example}
If $G$ is a path with two edges then $P_G$ is a line segment and $M_G$ only contains $P_G$.
\end{example}

\begin{example}\label{EX:Path}
If $G$ is the path with three edges then $P_G$ is a triangle and $M_G$ consists of $P_G$ and one of the edges in the triangle. The two elements in $M_G$ can be matched to give an acyclic matching in the Hasse diagram of $M_G$ with no critical elements.
\end{example}

\begin{example}
If $G$ is the cycle with four vertices then $P_G$ is a square and $M_G$ only contains $P_G$.
\end{example}

It is possible to understand the fundamental edges and vertices.

\begin{proposition}\label{Prop:FVertex}
Let $G$ be a graph so that $\overline{G}$ is triangle-free. Let $v\in V_d(\overline{G})$. The graph $C_G(v)$ is a clique with $d$ vertices. In particular $v$ is fundamental in $G$ if and only if $v\in V(G)\setminus (V_1(G)\cup V_2(G))$.
\end{proposition}

\begin{proof}
The graph $C_G(v)$ is a clique as $\overline{G}$ is triangle-free, the clique has a bipartite component if and only if it has one or two elements. The graph $N_G(v)$ is connected as it is a star.
\end{proof}

The criterion for a vertex to be fundamental is sometimes called the degree criterion for fundamentality.

\begin{proposition}\label{Prop:FEdge}
Let $G$ be a graph so that $\overline{G}$ is triangle-free and let $uv\in E(\overline{G})$. The graph $C_G(uv)$ has no vertices and $NC_G(uv)=N_G(u)\cup N_G(v)$. In particular $uv$ is fundamental if and only if $u$ and $v$ have a common neighbor in $G$.
\end{proposition}

\begin{proof}
The graph $C_G(uv)$ has no vertices as $\overline{G}$ is triangle-free. The graph $N_G(u)\cup N_G(v)$ is connected if and only if $u$ and $v$ have a common neighbor in $G$.
\end{proof}

\section{The graph $F(\overline{G},S_G)$ and $M_G$}\label{Sec:All}
This goal of this section is to describe $M_G$ when $G$ is connected and not bipartite and $\overline{G}$ is triangle-free. The main result essentially states that $M_G$ is isomorphic to the dualization of the face poset of $F(\overline{G},S_G)$. That $F(\overline{G},S_G)$ is well defined follows from the following proposition.

\begin{proposition}
Let $G$ be connected and not bipartite and $\overline{G}$ triangle-free. All vertices in $V_1(G)\cup V_2(G)$ are endpoints of edges in $S_G$.
\end{proposition}
\begin{proof}
If $v\in V_1(G)$ then there is a unique edge $uv\in E(\overline{G})$ containing $v$. Now $u$ and $v$ have a common neighbor in $G$ if and only if $u$ is not adjacent to all other vertices in $\overline{G}$. If $u$ is a neighbor to all other vertices in $\overline{G}$ then $u$ is isolated in $G$, this can not happen and then $uv$ is fundamental.

Similarly if $v\in V_2(G)$ then the edges $uv$ and $vw$ containing $v$ are both not fundamental if and only if both $v$ and $w$ are neighbors in $\overline{G}$ to everything in $V(\overline{G})\setminus \{u,v,w\}$. If both $v$ and $w$ are neighbors in $\overline{G}$ to everything in $V(\overline{G})\setminus \{u,v,w\}$ then $\overline{G}$ is a complete bipratite graph $G$ is disconnected.
\end{proof}

To prove the stated description of $M_G$  it is first necessary to determine all pairs of facets in $M_G$ whose intersection is a face in $M_G$. In order to do this all possible pairs are divided into types depending on some combinatorial data.

There are eleven different types of combinatorial pairs of fundamental sets in $G$ where $\overline{G}$ is triangle-free and $G$. The list is as follows.

\begin{enumerate}
\item\label{1} Two sets $\{u\}$ and $\{v\}$ where $uv\in E(\overline{G})$ is fundamental.
\item\label{A} Two sets $\{u\}$ and $\{v\}$ where $uv\in E(\overline{G})$ is not fundamental.
\item\label{3}  Two sets $\{u\}$ and $\{v\}$ where $uv\in E(G)$.
\item\label{4}  Two disjoint two element sets $\{u,v\}$ and $\{u',v'\}$.
\item\label{5}  Two not disjoint two element sets $\{u,v\}$ and $\{u,w\}$ where $u$ is fundamental.
\item\label{B} Two not disjoint two element sets $\{u,v\}$ and $\{u,w\}$ where $u$ is not fundamental.
\item\label{C} A two element set $\{u,v\}$ and a singleton $\{v\}$.
\item\label{8} A two element set $\{u,v\}$ and a singleton $\{w\}$ where $w$ is not adjacent to either $u$ or $v$ in $\overline{G}$.
\item\label{9} A two element set $\{u,v\}$ and a singleton $\{w\}$ where $w$ is adjacent $u$ but not $v$ in $\overline{G}$. The edge $uw$ is fundamental.
\item\label{10} A two element set $\{u,v\}$ and a singleton $\{w\}$ where $w$ is adjacent $u$ but not $v$ in $\overline{G}$. The edge $uw$ not fundamental. The vertex $u$ is fundamental.
\item \label{D} A two element set $\{u,v\}$ and a singleton $\{w\}$ where $w$ is adjacent $u$ but not $v$ in $\overline{G}$. The edge $uw$ not fundamental. The vertex $u$ is not fundamental.
\end{enumerate}

The following list of propositions show exactly what types of pairs from the list give intersections $P_{NC_G(U)}\cap P_{NC_G(U')}$ in $M_G$. The proofs either demonstrate an isolated vertex in $NC_G(U)\cap NC_G(U)$ or give edges to all vertices.

\begin{proposition}
Let $\{u\}$ and $\{v\}$ be a pair of fundamental sets of type \ref{1} then $P_{NC_G(u)}\cap P_{NC_G(v)}$ is not in $M_G$.
\end{proposition}
\begin{proof}
As $uv$ is fundamental there is some vertex $w$ adjacent to both $u$ and $v$ in $G$. Now $w$ is isolated in $N_G(u)\cap N_G(v)$.
\end{proof}

\begin{proposition}\label{prop:a}
Let $\{u\}$ and $\{v\}$ be a pair of fundamental sets of type \ref{A} then $P_{NC_G(u)}\cap P_{NC_G(v)}$ is in $M_G$.
\end{proposition}
\begin{proof}
As $uv$ is not fundamental all other vertices are adjacent to either $u$ or $v$ in $\overline{G}$, and no vertex is adjacent to both. This property carries over to $G$. In particular $NC_G(u)\cap NC_G(v)=N_G(u)\cup N_G(v)$ as $u$ and $v$ are not adjacent in $G$. None of the stars $N_G(u)$ and $N_G(v)$ have isolated vertices as $G$ is connected and not bipartite.
\end{proof}

\begin{proposition}
Let $\{u\}$ and $\{v\}$ be a pair of fundamental sets of type \ref{3} then $P_{NC_G(u)}\cap P_{NC_G(v)}$ is not in $M_G$.
\end{proposition}
\begin{proof}
The proof is split into two parts if $u$ and $v$ have a common neighbor in $G$ or not.

If $u$ and $v$ have a common neighbor $w$ in $G$ then $w$ is isolated in $N_G(u)\cap N_G(v)$.

The graphs $C_G(u)\cap N_G(v)$ and $C_G(v)\cap N_G(u)$ have no edges as all edges in $N_G(u)$ go to $u$ while $u\notin V(C_G(v))$ and similarly for  $N_G(v)$ and $C_G(u)$. If $u$ and $v$ have no common neighbor then at least one of the graphs $C_G(u)\cap N_G(v)$ and $C_G(v)\cap N_G(u)$ have a vertex $w$ as $G$ is not bipartite. Now $w$ is isolated in $NC_G(u)\cap NC_G(v)$.
\end{proof}

\begin{proposition}
Let $\{u,v\}$ and $\{u',v'\}$ be a pair of fundamental sets of type \ref{4} then $P_{NC_G(uv)}\cap P_{NC_G(u'v')}$ is not in $M_G$.
\end{proposition}
\begin{proof}
The graphs $C_G(uv)$ and $C_G(u'v')$ are empty as $\overline{G}$ is triangle-free. The edges in the graph $N_G(uv)\cap N_G(u'v')$ only go between elements of $\{u,v,u',v'\}$. Now all elements in $V(G)\setminus \{u,v,u',v'\}$ are isolated in $G$. The set $V(G)\setminus \{u,v,u',v'\}$ is not empty as $G$ is not bipartite.
\end{proof}

\begin{proposition}
Let $\{u,v\}$ and $\{u,w\}$ be a pair of fundamental sets of type \ref{5} then $P_{NC_G(uv)}\cap P_{NC_G(uw)}$ is not in $M_G$.
\end{proposition}
\begin{proof}
By the degree criterion for fundamentality there is some vertex $u'\in V(G)\setminus\{u,v,w\}$ not adjacent to $u$ in $G$. Now $u'$ is adjacent to both $v$ and $w$ in $G$ as $\overline{G}$ is triangle-free. Now $w$ is isolated in $N_G(uv)\cap N_G(uw)$.
\end{proof}

\begin{proposition}\label{prop:b}
Let $\{u,v\}$ and $\{u,w\}$ be a pair of fundamental sets of type \ref{B} then $P_{NC_G(uv)}\cap P_{NC_G(uw)}$ is in $M_G$.
\end{proposition}
\begin{proof}
By the degree criterion for fundamentality $u$ is adjacent to all vertices in $V(G)\setminus\{u,v,w\}$ in $G$. The set $V(G)\setminus\{u,v,w\}$ is not empty as $G$ is not bipartite. Now all vertices in $V(G)\setminus\{v,w\}$ are in some edge in $N_G(uv)\cap N_G(uw)$. Finally $vw$ is also an edge in $N_G(uv)\cap N_G(uw)$ as $\overline{G}$ is triangle-free.
\end{proof}

\begin{proposition}\label{prop:c}
Let $\{u,v\}$ and $\{v\}$ be a pair of fundamental sets of type \ref{C} then $P_{NC_G(uv)}\cap P_{NC_G(v)}$ is in $M_G$.
\end{proposition}
\begin{proof}
As $u$ is not a neighbor of $v$ in $G$ it follows that $N_G(uv)\cap N_G(v)=N_G(v)$. In particular $v$ and the neighbors of $v$ are not isolated in $NC_G(uv)\cap NC_G(v)$, the vertex $v$ has neighbors as $G$ is connected and not bipartite. The vertices not adjacent to $v$ in $G$ are adjacent to $u$ in $G$ as $\overline{G}$ is triangle-free, there are vertices in $V(G)\setminus \{u,v\}$ not adjacent to $v$ by the degree criterion for fundamentality. Now $N_G(uv)\cap C_G(v)$ connects all the vertices not in $N_G(v)$ to $u$.
\end{proof}

\begin{proposition}
Let $\{u,v\}$ and $\{w\}$ be a pair of fundamental sets of type \ref{8} then $P_{NC_G(uv)}\cap P_{NC_G(w)}$ is not in $M_G$.
\end{proposition}
\begin{proof}
The graph $C_G(w)\cap N_G(uv)$ has no edges as $u$ and $v$ are adjacent to $w$ in $G$. As $w$ is fundamental there is some vertex $w'\in V(G)\setminus \{u,v,w\}$ not adjacent to $w$ in $G$. Now $w'$ is isolated in $C_G(v)\cap N_G(uv)$.
\end{proof}

\begin{proposition}
Let $\{u,v\}$ and $\{w\}$ be a pair of fundamental sets of type \ref{9} then $P_{NC_G(uv)}\cap P_{NC_G(w)}$ is not in $M_G$.
\end{proposition}
\begin{proof}
As $uw$ is fundamental there is some vertex $w'\in V(G)\setminus \{u,v,w\}$ adjacent to both $u$ and $w$ in $G$. Now $w'$ is isolated in $N_G(uv)\cap N_G(w)$.
\end{proof}

\begin{proposition}
Let $\{u,v\}$ and $\{w\}$ be a pair of fundamental sets of type \ref{10} then $P_{NC_G(uv)}\cap P_{NC_G(w)}$ is not in $M_G$.
\end{proposition}
\begin{proof}
By the degree criterion for fundamentality there is a vertex $u'\in V(G)\setminus\{u,v,w\}$ not adjacent to $u$ in $G$. As $\overline{G}$ is triangle-free $u'$ is adjacent to both $v$ and $w$ in $G$. Now $u'$ is isolated in $N_G(uv)\cap N_G(w)$.
\end{proof}

\begin{proposition}\label{prop:d}
Let $\{u,v\}$ and $\{w\}$ be a pair of fundamental sets of type \ref{D} then $P_{NC_G(uv)}\cap P_{NC_G(w)}$ is in $M_G$.
\end{proposition}
\begin{proof}
All vertices in $V(G)\setminus \{u,v,w\}$ are adjacent to $w$ in $\overline{G}$ as $u$ is not fundamental and $uv$ is fundamental. Now $C_G(w)\cap N_G(uv)$ is a star with central vertex $u$ and leafs $V(G)\setminus \{u,v,w\}$. The set $V(G)\setminus \{u,v,w\}$ is not empty as $G$ is not bipartite.

The edge $uv$ is in $N_G(w)\cap C_G(uv)$ and no vertex is isolated in $NC_G(uv)\cap NC_G(w)$
\end{proof}

The set of facets of $P_G$ in $M_G$ is $\{P_{NC_G(v)}\mid v\in V(\overline{G})\setminus (V_1(\overline{G})\cup V_2(\overline{G}))\}\cup \{P_{NC_G(uv)}\mid uv\in S_G\}$. The identifications $v\leftrightarrow P_{NC_G(v)}$ and $uv\leftrightarrow P_{NC_G(uv)}$ are used in the following lemma and in this way the facets in $M_G$ is viewed as the set $S_G\cup V(\overline{G})\setminus (V_1(\overline{G})\cup V_2(\overline{G}))$.

\begin{lemma}\label{lemma:main}
Let $G$ be a connected and not bipartite graph. Let $\overline{G}$ be triangle-free. There are not codimension three faces in $M_G$. The poset $M_G$ is the dualization of the face poset of $F(\overline{G},S_G)$
\end{lemma}
\begin{proof}
First it is proved that no codimension three face is in $M_G$. Following the proofs of the propositions \ref{prop:a}, \ref{prop:b}, \ref{prop:c} and \ref{prop:d} the codimension two faces in $M_G$ are the edge polytopes of disjoint unions of stars. In particular any further intersections will isolate edges in the underlying graphs and then there can be no codimension three face.

Now $M_G$ is the face poset of some one-dimensional complex with the same vertex set as $F(\overline{G},S_G)$. The final step is to show that the adjacencies agree.

 The adjacencies come from pairs of type \ref{A}, \ref{B}, \ref{C} and \ref{D}. This agrees with the adjacency description of $F(\overline{G},S_G)$. Two vertices in $V(\overline{G})\setminus (V_1(\overline{G})\cup V_2(\overline{G}))$ are adjacent if they are adjacent in $\overline{G}$ but the edge between them is not in $S_G$, this comes from pairs of type \ref{A}. Two edges in $S_G$ are adjacent if they have a common endpoint in $V_1(\overline{G})\cup V_2(\overline{G})$, this comes from pairs of type \ref{B}. Vertices in $V(\overline{G})\setminus (V_1(\overline{G})\cup V_2(\overline{G}))$ are adjacent to the edges in $S_G$ containing them, this comes from pairs of type \ref{C}. Finally a vertex $u$ in $V(\overline{G})\setminus (V(\overline{G})\cup V(\overline{G}))$ is adjacent to an edge $vw$ in $S_G$ not containing $u$ if $uv\in E(\overline{G})\setminus S_G$, this comes from pairs of type \ref{D}.
\end{proof}

The point of Lemma \ref{lemma:main} is that all parts of the hull resolution of $I_G$ now can be understood in terms of graphs homotopy equivalent to $\overline{G}$ or induced subgraphs of $\overline{G}$. For cycles $\overline{G}$ all components of proper subgraphs $G[U]$ are paths and $M_{G[U]}$ has an acyclic matching where first all components are contracted to vertices and then one of the critical vertices can be matched to $P_{G[U]}$ thus giving a resolution as described in the introduction. This argument generalizes to prove Theorem \ref{thm:main}.

\section{Proof of Theorem \ref{thm:main}.}\label{Sec:Main}

First do the only if part.

Let $\overline{G}[U]$ be two cycles. Now the cycles are not triangles as $\overline{G}$ is triangle-free. The graph $G[U]$ is connected and not bipartite and Lemma \ref{lemma:main} applies to $M_{G[U]}$.

Consider ordinary discrete Morse theory for the one-dimensional complex consisting of two disjoint cycles. The best possible resulting complex consists of two vertices with loops. For $M_{G[U]}$ this corresponds to a partial matching where $P_{G[U]}$ is critical and there are two critical facets and they each have a critical face of codimension two attached. Only one of the facets can be matched to $P_{G[U]}$ and the other critical facet remain with the critical codimension two face attached and the resolution is not minimal.

Let $\overline{G}$ be a graph with no induced subgraph isomorphic to the disjoint union of cycles. Let $G[U]$ be an induced subgraph without isolated vertices. Let $G_1\ldots, G_n$ be the components of $G[U]$. Now $M_{G[U]}=\prod_{i\in [n]} M_{G_i}$ and either Lemma \ref{lemma:main} applies to $G_i$ or $G_i$ is bipartite. In the bipartite case there are acyclic matchings of $M_{G_i}$ with at most one critical cell.

If Lemma  \ref{lemma:main} applies to $G_i$ consider again ordinary discrete Morse theory. The optimal acyclic matching contracts all components without cycles to single vertices without loops and if there is a component with cycles it is unique and it is contracted to a vertex with potentially many loops.  For $M_{G_i}$ this translates to a partial matching where $P_{G_i}$ is critical and there is a critical facet for each component and at most one of them has critical codimension two faces attached. If there is component with cycles then the corresponding facet is matched to $P_{G_i}$ otherwise any of the facets can be matched to $P_{G_i}$. 

Now there are matchings for each $M_{G_i}$ and these glue together to a matching of $\prod_{i\in [n]}M_{G_i}$ in the following way.

For every matched pair $\sigma_1,\tau_1\in M_{G_2}$ match  $\sigma_1\times\sigma_2\times\cdots\times \sigma_n$ to $\tau_1\times\sigma_2\times\cdots\times \sigma_n$. Now there might be some critical cell $\phi_1\in M_{G_1}$. Proceed to for every matched pair $\sigma_2,\tau_2\in M_{G_2}$  match $\phi_1\times\sigma_2\times \sigma_3\times\cdots\times \sigma_n$ to $\phi_1\times\tau_2\times \sigma_3\times\cdots\times \sigma_n$ and so on. In the end the critical cells are of the form $\phi_1\times\cdots\times\phi_n$ where each $\phi_i$ is critical in $M_{G_i}$. 

This way of constructing acyclic matchings is standard in discrete Morse theory but an argument for acyclicity is given anyway as the argument is also used to show minimality of the resulting resolution.

Assume that there is a directed path in the Hasse diagram of $\prod_{i\in [n]}M_{G_i}$ where the matched edges are reversed. Suppose one of the edges in the cycle is from $\sigma_1\times \sigma_2\times\cdots\sigma_{n}$ to $\sigma_1\times \sigma_2\times\cdots\sigma_{i-1}\times\tau_i\times \sigma_{i+1}\times\cdots\sigma_{n}$ so that the edge comes from the edge from $\sigma_i$ to $\tau_i$ in the reversed Hasse diagram of $M_{G_i}$. Following the cycle it is possible to get back to $\sigma_1\times \sigma_2\times\cdots\sigma_{n}$ and then the edges from $M_{G_i}$ give a cycle and this is a contradiction.

What remains is to show that the resulting resolution is minimal. In order to do this a slight strengthening of Theorem \ref{thm:adm} is needed including a sufficient condition for minimality. One such condition is that the resulting resolution is minimal if there is no directed path between critical cells with the same label in the graph obtained by reversing the matched edges in the Hasse diagram. This condition is a special case of Lemma 7.5 in \cite{BW}.

It is enough to consider directed paths where all cells have the same label as the matching is homogenous.

Consider the Hasse diagrams of $M_{G_i}$ where the matched edges are reversed. There can not be any directed path between critical cells with the same dimension as the endpoint then have to be matched. Any directed path between critical cells then have to go from a facet to a codimension two face. In particular the path has to pass through $P_{G_i}$ as the critical elements from codimension two faces correspond to edges in a single component of $\overline{G}_i$ and the critical facets correspond to the other components. Now there is no path between critical faces of $M_{G_i}$ as the facet matched to $P_{G_i}$ is a sink. In the same way as acyclicity of the matchings for each $M_{G_i}$ gives acyclicity of the matching of $\prod_{i\in [n]}M_{G_i}$ this argument extends to show that there is no directed path between critical cells in the Hasse diagram of $\prod_{i\in [n]}M_{G_i}$ with matched edges reversed.

\end{document}